\documentclass[a4paper]{article}
\pdfoutput=1
\usepackage{amsmath,amssymb,amsthm}
\usepackage[english]{babel}
\usepackage{hyperref}
\usepackage[textsize=small,textwidth=1.3in]{todonotes}
\usepackage{tikz}
\usepackage{xcolor}

\theoremstyle{plain}
\newtheorem{theorem}{Theorem}
\newtheorem{proposition}[theorem]{Proposition}
\newtheorem{lemma}[theorem]{Lemma}

\theoremstyle{definition}
\newtheorem{definition}[theorem]{Definition}

\newtheorem*{question}{Main Question}

\title{Relations between M\"obius and coboundary polynomial}
\author{Relinde Jurrius}

\begin{document}

\maketitle

\begin{abstract}
It is known that, in general, the coboundary polynomial and the M\"obius polynomial of a matroid do not determine each other. Less is known about more specific cases. In this paper, we will try to answer if it is possible that the M\"obius polynomial of a matroid, together with the M\"obius polynomial of the dual matroid, define the coboundary polynomial of the matroid. In some cases, the answer is affirmative, and we will give two constructions to determine the coboundary polynomial in these cases.
\end{abstract}

\section{Introduction}

When studying invariant polynomials of matroids, much attention is given to the Tutte polynomial. Much of other polynomials associated to graphs, arrangements, linear codes and matroids turn out to be an evaluation of the Tutte polynomial, or define the Tutte polynomial. Sometimes the polynomials and the Tutte polynomial determine each other. \\
The latter is, for simple matroids, the case with the coboundary polynomial. A polynomial that did not attract too many attention, is the M\"obius polynomial. It 
is not equivalent to the Tutte polynomial. \\
It follows that, in general, the coboundary polynomial and the M\"obius polynomial do not determine each other.  Less is known about more specific cases. In this paper, we will try to answer if it is possible that the M\"obius polynomial of a matroid, together with the M\"obius polynomial of the dual matroid, define the coboundary polynomial of the matroid. In some cases, the answer is affirmative, and we will give two constructions to determine the coboundary polynomial in these cases. \\
Much of the theory we use for matroids, has originated from coding theory. For understanding the results about matroid theory, it is not necessary to know this origin, but it is included in this paper as a motivation for the techniques we use. This material is found in Section \ref{ewe} and the beginning of Section \ref{zeta}.

\section{Extended weight enumerator}\label{ewe}

Let $C$ be a linear $[n,k]$ code over $\mathbb{F}_q$ with generator matrix $G$. Then we can form the $[n,k]$ code $C\otimes\mathbb{F}_{q^m}$ over $\mathbb{F}_{q^m}$ by taking all $\mathbb{F}_{q^m}$-linear combinations of the codewords in $C$. We call this the \emph{extension code} of $C$ over $\mathbb{F}_{q^m}$. By embedding its entries in $\mathbb{F}_{q^m}$, we find that $G$ is also a generator matrix for the extension code $C\otimes\mathbb{F}_{q^m}$. This motivates the usage of $T$ as a variable for $q^m$ in the next definition.
\begin{definition}
The \emph{extended weight enumerator} is the polynomial
\[ W_C(X,Y,T)= \sum_{w=0}^{n} A_w(T) X^{n-w}Y^w \]
where the $A_w(T)$ are integral polynomials in $T$ and $A_w(q^m)$ is the number of codewords of weight $w$ in $C\otimes\mathbb{F}_{q^m}$.
\end{definition}
See \cite{jurrius:2011a} for a proof that the $A_w(T)$ are indeed polynomials of degree at most $k$. \\

To every linear code $C$ we can associate a matroid, represented by the columns of the generator matrix $G$. Since $C$ and all its extension codes $C\otimes\mathbb{F}_{q^m}$ have the same generator matrix, they also have the same matroid associated to it. It turns out that the extended weight enumerator is completely determined by the Tutte polynomial, and vice versa. Therefore, we can extend the definition of the extended weight enumerator from codes to matroids in general. See \cite{jurrius:2011a} for details.

\section{Matroids and their polynomials}

For an excellent introduction into the topic of matroids, see \cite{welsh:1976} or \cite{oxley:2011}. More about the theory of geometric lattices and the M\"obius function can be found in \cite{aigner:1979,jurrius:2011a,stanley:2007}. In \cite{white:1986} the cryptomophism between matroids and geometric lattices is discussed. \\

For a matroid $M$ with rank function $r$ and dual matroid $M^*$, we will study the following parameters:
\begin{itemize}
\item $n$, the number of elements of $M$ and $M^*$;
\item $k$, the rank of $M$;
\item $d$, the size of the smallest cocircuit in $M$ (i.e., circuit in $M^*$);
\item $d^*$, the size of the smallest circuit in $M$ (i.e., cocircuit in $M^*$).
\end{itemize}
The reason to study $d$ and $d^*$ comes from coding theory. If a matroid is representable over a finite field, there is  linear code associated to it with minimum distance $d$ and dual minimum distance $d^*$. \\

Throughout this paper, we will restrict ourselves to simple matroids, i.e., matroids that do not have loops or parallel elements. Also the dual of a matroid is assumed to be simple. This implies $d>2$ and $d^*>2$. In this case, there is a two-way equivalence between matroids and geometric lattices: we will freely change between these objects when necessary. So in the following definition of the coboundary polynomial, it would have made sense to talk about the coboundary polynomial of a matroid, but for the definition the setting of geometric lattices is more convenient.

\begin{definition}
Let $L$ be a geometric lattice. The \emph{two variable characteristic} or \emph{coboundary} polynomial in the variable $S$ and $T$ is given by
\[ \chi _L(S,T)=\sum_{x\in L}\ \sum_{x\leq y\in L}\mu(x,y)\, S^{a(x)}T^{r(L)-r(y)}, \]
where $a(x)$ is the number of atoms $a$ in $L$ such that $a \leq x$, and $r(y)$ is the rank of $y$ in $L$.
\end{definition}

For every matroid $M$, not necessarily simple, the coboundary polynomial of the associated lattice $L(M)$ is determined by the Tutte polynomial. For simple matroids, this is a two way equivalence: the coboundary polynomial of $L(M)$ determines the Tutte polynomial of the matroid $M$. This makes it possible to prove the following theorem:

\begin{theorem}\label{chi-dual}
Let $\chi_M(S,T)$ be the coboundary polynomial of a simple matroid $M$ with simple dual $M^*$. Let $\chi_{M^*}(S,T)$ be the coboundary of $M^*$. Then
\[ \chi_{M^*}(S,T)=(S-1)^nT^{-r(M)}\chi_M\left(\frac{S+T-1}{S-1},T\right). \]
\end{theorem}
\begin{proof}
Use the fact that the Tutte polynomial of a matroid and its dual completely determine each other, and apply the above correspondence with the coboundary polynomial.
\end{proof}

One might notice the resemblance between this theorem and the MacWilliams relations from coding theory. In fact, the coboundary polynomial of a simple matroid is the reciprocal inhomogeneous form of the extended weight enumerator of this matroid:
\[ \chi_M(S,T)=S^nW_M(1,S^{-1},T). \]
This means $\chi_i(T)=A_{n-i}(T)$. For more details, see \cite{jurrius:2011a}. \\

The second polynomial we study, is the M\"obius polynomial. It is defined on geometric lattices, but we can consider this also as a definition for simple matroids.

\begin{definition}
Let $L$ be a geometric lattice. The two variable {\em M\"obius} polynomial in the variable $S$ and $T$ is given by
\[ \mu _L(S,T)=\sum_{x\in L}\ \sum_{x\leq y\in L}\mu(x,y)\, S^{r(x)}T^{r(L)-r(y)}. \]
\end{definition}

For convenience, we often refer to the coboundary and M\"obius polynomial in the following form:
\[ \chi_M(S,T)=\sum_{i=0}^k\chi_i(T)\,S^i,\qquad\mu_M(S,T)=\sum_{i=0}^k\mu_i(T)\,S^i. \]
The polynomial $\chi_i(T)$ is sometimes referred to as the \emph{$i$-th defect polynomial}.

\section{Connections}

Some natural questions arise about the dependencies between the coboundary polynomial and M\"obius polynomial of a matroid and its dual. First of all, do the coboundary and M\"obius polynomial determine each other? The answer is ``no'', even if both the matroid and its dual are simple. Counterexamples can be found in \cite{jurrius:2011a}, Examples 58 and 60. \\
In Theorem \ref{chi-dual} we saw that the coboundary polynomials of a matroid and its dual are completely determined by each other. Does such a formula also exists for the M\"obius polynomial? To answer this, we need some more theory.
\begin{lemma}\label{rank-size}
Let $M$ be a matroid. Then for all elements $x\in M$ with $r(x)<d^*-1$, we have $|x|=r(x)$. Furthermore, if $M$ is simple, we have $a(x)=r(x)$ in the corresponding geometric lattice.
\end{lemma}
\begin{proof}
By definition, $d^*$ is the size of the smallest circuit in $M$ and thus the size of the smallest dependent set in $M$. It has rank $d^*-1$. This means all elements $x\in M$ of rank $r(x)<d^*-1$ are independent and have $|x|=r(x)$. For simple matroids, $|x|=a(x)$ in the corresponding geometric lattice.
\end{proof}
\begin{proposition}\label{mu-d*}
Given the M\"obius polynomials $\mu_M(S,T)$ of a matroid. Then we can determine the parameter $d^*$ of the matroid $M$.
\end{proposition}
\begin{proof}
The coefficient of the term $S^iT^j$ in the M\"obius polynomial is given by
\[ \sum_{\substack{x\in L \\ r(x)=i}}\ \sum_{\substack{x\in L \\ r(x)=i}} \mu(x,y). \]
These numbers are also known as the \emph{doubly-indexed Whitney numbers of the first kind}. In the case $j=k-i$, we just count the number of elements in $L$ of rank $i$, i.e., the number of flats of rank $i$ in $M$. From Lemma \ref{rank-size} it now follows that for $i<d^*-1$ all elements of rank $i$ are flats, so there are ${n\choose i}$ of them. For $i\geq d^*-1$, the number of flats is strictly smaller then ${n\choose i}$. Therefore we can determine $d^*$ from the M\"obius polynomial of $M$.
\end{proof}
In the previously mentioned Example 58 in \cite{jurrius:2011a}, we have two matroids with the same M\"obius polynomial but with different $d$. By Proposition \ref{mu-d*}, this means that their duals cannot have the same M\"obius polynomial. This gives a negative answer to the question in \cite[\S10.5]{jurrius:2011a} if the M\"obius polynomial of a matroid and its dual are determined by each other. \\

To summarize, together with Theorem \ref{chi-dual} we know the following about the coboundary and M\"obius polynomials of a matroid and its dual:
\begin{itemize}
\item The coboundary polynomial $\chi_M(S,T)$ of a matroid and the coboundary polynomial $\chi_{M^*}(S,T)$ of the dual matroid completely determine each other.
\item The M\"obius polynomial $\mu_M(S,T)$ of a matroid does not determine the M\"obius polynomial $\mu_{M^*}(S,T)$ of the dual matroid.
\item The coboundary polynomial $\chi_M(S,T)$ does not determine the M\"obius polynomial $\mu_M(S,T)$. The same holds in the dual case.
\item The M\"obius polynomial $\mu_M(S,T)$ does not determine the coboundary polynomial $\chi_M(S,T)$.
\end{itemize}

The last three statements also hold in case $M$ and/or $M^*$ are not simple. In this paper, we will address another question between dependencies:
\begin{question}
Given the M\"obius polynomials $\mu_M(S,T)$ and $\mu_{M^*}(S,T)$ of a matroid and its dual. Do they determine $\chi_M(S,T)$?
\end{question}

We will see that, in some cases, the answer is ``yes''. Proposition \ref{mu-d*} tells us that the M\"obius polynomial gives us information about the dual of the matroid. This is the reason to ask if the M\"obius polynomial of the matroid, together with the M\"obius polynomial of its dual, determine the coboundary polynomial. \\
For completeness, note that $\mu_M(S,T)$ and $\mu_{M^*}(S,T)$ define not only $d^*$ and $d$, respectively, but also $n$ and $k$: the degree of $\mu_M(S,T)$ in $S$ is $r(M)=k$, and the degree of $\mu_{M^*}(S,T)$ in $S$ is $r(M^*)=n-k$.

\begin{theorem}\label{mu-chi}
Let $M$ be a matroid, and let the M\"obius polynomial $\mu_M(S,T)$ be given. Then part of the coboundary polynomial $\chi_M(S,T)$ is determined from this:
\[ \chi_i(T)=\left\{
\begin{array}{ll}
\mu_i(T) & \mbox{for }i<d^*-1 \\
0 & \mbox{for }n-d<i<n \\
1 & \mbox{for }i=n.
\end{array}\right. \]
\end{theorem}
\begin{proof}
The first equality follows from Proposition \ref{mu-d*}, the definition of the M\"obius and coboundary polynomial, and Lemma \ref{rank-size}. If $d$ is the smallest size of a cocircuit in $M$, then $n-d$ is the biggest size of a hyperplane in $M$ and thus the biggest size of a flat with rank smaller then $k$ in $M$. This implies the second equality. The third equality is obvious from the definition of the coboundary polynomial.
\end{proof}

Using this theorem, we can determine the value of $\chi_i(T)$ for $(d^*-1)+(d-1)+1=d^*+d-1$ values of $i$. This leaves $n+1-(d^*+d-1)=n-d-d^*+2$ of the $\chi_i(T)$ unknown. We can say the same about the coefficients $\chi_i^*(T)$ of the coboundary polynomial $\chi_{M^*}(S,T)$ of the dual matroid. The idea is to use Theorem \ref{chi-dual} to calculate the missing values of $\chi_i(T)$ and $\chi_i^*(T)$. We first rewrite Theorem \ref{chi-dual} to a more convenient form.

\begin{proposition}\label{chi-i-dual}
Let $\chi_i(T)$ be the coefficients of the coboundary polynomial of a simple matroid $M$ with simple dual $M^*$. Let $\chi_i^*(T)$ be the coefficients of the coboundary of $M^*$. Then
\[ T^{v-k}\sum_{i=v}^n{i\choose v}\chi_i(T)=\sum_{i=n-v}^n\chi_i^*(T),\qquad v=0,\ldots,n. \]
\end{proposition}
\begin{proof}
This is obtained by rewriting the formula in Theorem \ref{chi-dual}. This can be done in the same way as rewriting the MacWilliams relations from coding theory, see for example \cite[\S5.2]{macwilliams:1977}.
\end{proof}

In some cases, the relations from Theorem \ref{mu-chi} and Propositions \ref{chi-i-dual} are enough to completely determine the coboundary polynomial $\chi_M(S,T)$ from the polynomials $\mu_M(S,T)$ and $\mu_{M^*}(S,T)$.

\begin{theorem}\label{bound}
Let $M$ be a matroid with $2(d+d^*)\geq n+3$. Then the M\"obius polynomials $\mu_M(S,T)$ and $\mu_{M^*}(S,T)$ determine $\chi_M(S,T)$.
\end{theorem}
\begin{proof}
We try to determine the coboundary polynomials of $M$ and $M^*$ simultaneously. First we use Theorem \ref{mu-chi} for $M$ and $M^*$. This gives us the value of $\chi_i(T)$ for $i<d^*-1$ and $i>n-d$, and the value of $\chi_i^*(T)$ for $i<d-1$ and $i>n-d^*$. So we are left with the unknowns
\[ \chi_{d^*-1}(T),\chi_{d^*}(T),\ldots,\chi_{n-d}(T),\chi_{d-1}^*(T),\chi_d^*(T),\ldots,\chi_{n-d^*}^*(T). \]
This are $2(n-d-d^*+2)$ variables. Proposition \ref{chi-i-dual} gives us $n+1$ equations. In order for this system to be solvable, we need at least as may equations as unknowns. This means
\begin{eqnarray*}
n+1 & \geq & 2(n-d-d^*+2) \\
n+1 & \geq & 2n+4-2(d+d^*) \\
2(d+d^*) & \geq & n+3.
\end{eqnarray*}
We now need to show that, given $2(d+d^*)\geq n+3$, we have enough independent equations. Since all the  coefficients of the equations are known, it is possible to do this directly, but that gives lengthy calculations. We will give a more graphical approach. First, we visualize how Proposition \ref{chi-i-dual} looks like in matrix form. The grey areas are filled with nonzero entries, the white areas contain only zeros.
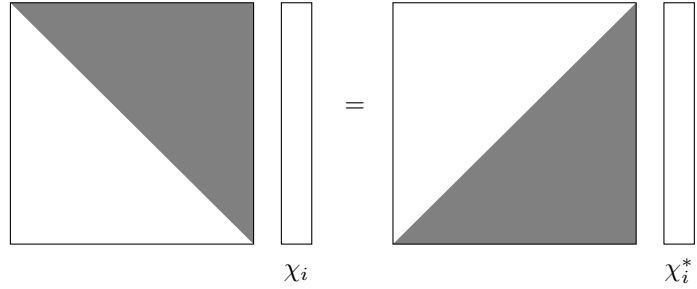
\[ \begin{array}{ccccc}
\begin{tikzpicture}[scale=.4]
\fill[gray] (0,8) -- (8,0) -- (8,8);
\draw (0,0) rectangle (8,8);
\end{tikzpicture} &
\begin{tikzpicture}[scale=.4]
\draw (0,0) rectangle (1,8);
\end{tikzpicture} & 
\begin{picture}(1,8)
\put(.2,5){$=$}
\end{picture} &
\begin{tikzpicture}[scale=.4]
\fill[gray] (0,0) -- (8,0) -- (8,8);
\draw (0,0) rectangle (8,8);
\end{tikzpicture} &
\begin{tikzpicture}[scale=.4]
\draw (0,0) rectangle (1,8);
\end{tikzpicture} \\
 & \chi_i & & & \chi_i^*
\end{array} \]
From the triangular shape of the matrices, it is clear that they both have full rank -- something we could have also concluded from the fact that the relation in Theorem \ref{chi-dual} is a two way equivalence. We order the system now in a way that all unknowns are on the left hand side. This means for the first matrix we ``cut off'' $d^*-1$ columns at the right of the matrix, and $d-1$ at the left, since they correspond to values of $i$ for which $\chi_i(T)$ is known. For the second matrix, it is the other way around. Since we assumed $2(d+d^*)\geq n+3$, we are cutting off at least half of the rows. The new system looks like this:
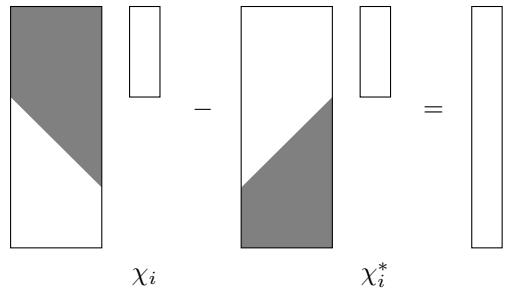
\[ \begin{array}{ccccccc}
\begin{tikzpicture}[scale=.4]
\fill[gray] (0,8) -- (3,8) -- (3,2) -- (0,5);
\draw (0,0) rectangle (3,8);
\end{tikzpicture} & 
\begin{tikzpicture}[scale=.4]
\draw[white] (0,0) -- (0,5);
\draw (0,5) rectangle (1,8);
\end{tikzpicture} & 
\begin{picture}(1,8)
\put(.2,5){$-$}
\end{picture} &
\begin{tikzpicture}[scale=.4]
\fill[gray] (0,0) -- (3,0) -- (3,5) -- (0,2);
\draw (0,0) rectangle (3,8);
\end{tikzpicture} & 
\begin{tikzpicture}[scale=.4]
\draw[white] (0,0) -- (0,5);
\draw (0,5) rectangle (1,8);
\end{tikzpicture} & 
\begin{picture}(1,8)
\put(.2,5){$=$}
\end{picture} &
\begin{tikzpicture}[scale=.4]
\draw (0,0) rectangle (1,8);
\end{tikzpicture} \\
 & \chi_i & & & \chi_i^* & &
\end{array} \]
The vector on the right hand side is known, and depends on $d$, $d^*$ and the two M\"obius polynomials. The matrices both have full rank $n-d-d^*+2$, as is clear from their shape. We can write this as one system by ``glueing together'' the matrices on the left hand side.
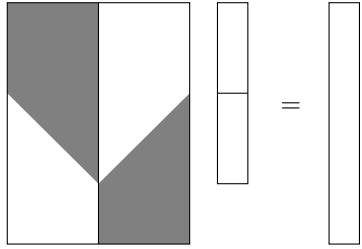
\[ \begin{array}{cccc}
\begin{tikzpicture}[scale=.4]
\fill[gray] (0,8) -- (3,8) -- (3,2) -- (0,5);
\fill[gray] (3,0) -- (6,0) -- (6,5) -- (3,2);
\draw (0,0) rectangle (6,8);
\draw (3,0) -- (3,8);
\end{tikzpicture} &
\begin{tikzpicture}[scale=.4]
\draw[white] (0,0) -- (0,2);
\draw (0,2) rectangle (1,8);
\draw (0,5) -- (1,5);
\end{tikzpicture} &
\begin{picture}(1,8)
\put(.2,5){$=$}
\end{picture} &
\begin{tikzpicture}[scale=.4]
\draw (0,0) rectangle (1,8);
\end{tikzpicture}
\end{array} \]
We need to show that this matrix has full rank. Have a look at the bottom $d$ rows of this matrix. The complete left side is zero, so we ignore that for a moment. The right side has all entries nonzero, and from Proposition \ref{chi-i-dual} we know the entries are binomial coefficients:
\[ \left( \begin{array}{cccc}
{d-1\choose d-1} & {d\choose d-1} & \cdots & {n-d^*\choose d-1} \\
\vdots & \vdots & & \vdots \\
{d-1\choose 1} & {d\choose 1} & \cdots & {n-d^*\choose 1} \\
{d-1\choose 0} & {d\choose 0} & \cdots & {n-d^*\choose 0}
\end{array} \right). \]
By the inductive relations between binomial coefficients, we can preform row operations on this matrix to obtain
\[ \left( \begin{array}{cccc}
0 & 0 & \cdots & {n-d-d^*+1\choose d-1} \\
\vdots & \vdots & & \vdots \\
0 & {1\choose 1} & \cdots & {n-d-d^*+1\choose 1} \\
{0\choose 0} & {1\choose 0} & \cdots & {n-d-d^*+1\choose 0}
\end{array} \right). \]
Flipping the matrix upside down, we have obtained the following picture:
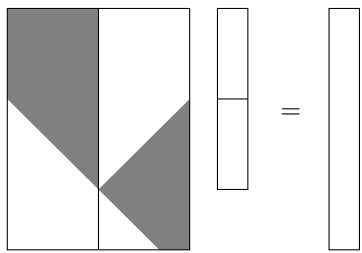
\[ \begin{array}{cccc}
\begin{tikzpicture}[scale=.4]
\fill[gray] (0,8) -- (3,8) -- (3,2) -- (0,5);
\fill[gray] (5,0) -- (6,0) -- (6,5) -- (3,2);
\draw (0,0) rectangle (6,8);
\draw (3,0) -- (3,8);
\end{tikzpicture} &
\begin{tikzpicture}[scale=.4]
\draw[white] (0,0) -- (0,2);
\draw (0,2) rectangle (1,8);
\draw (0,5) -- (1,5);
\end{tikzpicture} &
\begin{picture}(1,8)
\put(.2,5){$=$}
\end{picture} &
\begin{tikzpicture}[scale=.4]
\draw (0,0) rectangle (1,8);
\end{tikzpicture}
\end{array} \]
In this picture, we show the case for $d<n-d-d^*+2$. If we had $d\geq n-d-d^*+2$, we would have obtained a matrix that was of full rank and we were done. If $d^*\geq n-d-d^*+2$, we can change $M$ and $M^*$ and we are also done. So from now on, assume $d,d^*<n-d-d^*+2$. \\
Call the left and the right half of the matrix $L$ and $R$. Suppose a linear combination of the columns of the matrix is zero. Since all columns inside $L$ and inside $R$ are independent, this means we can make a linear combination $\mathbf{l}$ of columns of $L$ and a linear combination $\mathbf{r}$ of columns of $R$ that are both nonzero and a nonzero multiple of each other. \\
By the shape of $L$ and $R$, the first $d^*$ and the last $d$ entries of $\mathbf{l}$ and $\mathbf{r}$ have to be zero. We will show that the remaining $n-d-d^*+1$ entries of $\mathbf{l}$ and $\mathbf{r}$ cannot be multiples of each other. \\
Crucial in the proof is that all rows of $L$ are multiplied with a different power of $T$, whereas $R$ is completely filled with integers. Therefore, any linear combination of columns of $R$ will have the same powers of $T$ involved in every nonzero entry, even if we take the coefficients of the linear combination to be polynomials in $T$ and $T^{-1}$. On the other hand, the entries of $\mathbf{l}$ will all have different powers of $T$ involved. The only possibility to cancel this out, is if we can have only one nonzero entry in $\mathbf{l}$ and $\mathbf{r}$, at the same place. \\
We focus now on the matrix $L$. It has maximal (column) rank $n-d-d^*+2$. From Proposition \ref{chi-i-dual} we know the entries are binomial coefficients, with every row multiplied with another (possibly negative) power of $T$. The first $d^*$ rows form a matrix with rank $d^*$, from the same reasoning we used for the last $d$ rows of $R$. So if we make a linear combination of the columns of $L$ where the first $d^*$ entries are zero, there are $n-d-d^*+2-d^*=n-d-2d^*+2$ free variables involved. Notice we assumed $d^*<n-d-d^*+2$, so this number is positive. We can use those free variables to make more entries of $\mathbf{l}$ zero: add one of the middle $n-d-d^*+1$ rows of $L$ as an extra constraint, and choose one of the free variables in a way that the corresponding entry in $\mathbf{l}$ becomes zero. We are left with
\[ n-d-d^*+1-(n-d-2d^*+2)=d^*-1\geq2 \]
entries of $\mathbf{l}$ that are not zero. They also cannot be zero ``by accident'' since the middle $n-d-d^*+1$ rows of $L$ form a matrix of full rank. So $\mathbf{l}$ cannot have only one nonzero entry, as was to be shown. \\
To summarize, we have shown that we can use Theorem \ref{mu-chi} and some of the equations in Proposition \ref{chi-i-dual} to find $\chi_M(S,T)$ from $\mu_M(S,T)$ and $\mu_{M^*}(S,T)$ if $2(d+d^*)\geq n+3$.
\end{proof}

\section{Alternative approach: zeta polynomials}\label{zeta}

The two-variable zeta polynomial is extensively studied by Duursma \cite{duursma:2004}, who defined and studied the one-variable case in \cite{duursma:1999,duursma:2001a}. We start with the definition from coding theory, to motivate the case of the coboundary polynomial. The definitions only hold for codes with minimum distance and dual minimum distance at least 3; so the corresponding matroids are simple. \\
For the reader not familiar with coding theory, it is possible to directly take Theorem \ref{uniform-basis} as a definition for the two-variable zeta polynomial.
\begin{definition}\label{def-zeta}
Let $C$ be a linear $[n,k,d]$ code over $\mathbb{F}_q$ with extended weight enumerator $W_C(X,Y,T)$. The \emph{two-variable zeta polynomial} $P_C(Q,T)$ of this code is the unique polynomial of degree at most $n-d$ in $S$ such that the generating function
\[ \frac{P_C(Q,T)}{(1-Q)(1-TQ)}(Y(1-Q)+XQ)^n \]
has expansion
\[ \ldots+\frac{W_C(X,Y,T)-X^n}{T-1}Q^{n-d}+\ldots. \]
The quotient $Z_C(Q,T)=P_C(Q,T)/((1-Q)(1-TQ))$ is called the \emph{two-variable zeta function}.
\end{definition}
The two-variable zeta polynomial and the extended weight enumerator determine each other, see \cite{duursma:2001a,duursma:2004}. The original definition gives a polynomial $P(T,u)$ in the variables $T$ and $u$ instead of $Q$ and $T$, respectively. We change this to keep consistency with existing literature on the extended weight enumerator. Just as with the other polynomials, we often refer to the zeta polynomial in the following form:
\[ P_C(Q,T)=\sum_{i=0}^rP_i(T)\,Q^i. \]

The extended weight enumerator of an MDS code is completely determined by its parameters (see for example \cite{macwilliams:1977},\cite{jurrius:2011a}). So even if there does not exist an MDS code with parameters $[n,k,d]$, we can formally define its extended weight enumerator $M_{n,d}$. The coefficient of $X^{n-w}Y^w$ is
\[ {n\choose w}(T-1)\sum_{t=0}^{w-d}(-1)^t{w-1\choose t}T^{w-d-t}. \]
\begin{proposition}\label{zeta-MDS}
A code is MDS if and only if $P_C(Q,T)=1$.
\end{proposition}
\begin{proof}
Since we know how the extended weight enumerator of an MDS code looks like, we can proof this Proposition by writing out the generating function from Definition \ref{def-zeta} with $P_C(Q,T)=1$ and see that the coefficient of $Q^{n-d}$ indeed gives the extended weight enumerator of an MDS code.
\end{proof}
\begin{theorem}\label{MDS-basis}
The zeta polynomial gives us a way to write the extended weight enumerator on a basis of MDS weight enumerators:
\[ W_C(X,Y,T)=P_0(T)\,M_{n,d}+P_1(T)\,M_{n,d+1}+\ldots+P_r(T)\,M_{n,d+r}. \]
\end{theorem}
\begin{proof}
This follows directly from Definition \ref{def-zeta} and Proposition \ref{zeta-MDS}.
\end{proof}
Duursma \cite{duursma:2004} extended the definition of the zeta polynomial to matroids. Similar to that approach, we can use that we already extended the definition of the extended weight enumerator to matroids.
\begin{theorem}
Let $M$ be a matroid with coboundary polynomial $\chi_M(S,T)$. The \emph{two-variable zeta polynomial} $P_M(Q,T)$ of this matroid is the unique polynomial of degree at most $n-d$ in $Q$ such that the generating function
\[ \frac{P_M(Q,T)}{(1-Q)(1-TQ)}(1+(S-1)T)^n \]
has expansion
\[ \ldots+\frac{\chi_M(S,T)-S^n}{T-1}Q^{n-d}+\ldots. \]
\end{theorem}
\begin{proof}
Apply $X=1$ and $Y=S^{-1}$ in the definition of the zeta function and multiply the whole equation with $S^n$.
\begin{eqnarray*}
Z(Q,T)\cdot(Y(1-Q)+XQ)^n & = & \ldots+\frac{W_C(X,Y,T)-X^n}{T-1}Q^{n-d}+\ldots \\
Z(Q,T)\cdot S^n\cdot(S^{-1}(1-Q)+Q)^n & = & \ldots+\frac{W_C(1,S^{-1},T)-1}{T-1}S^nQ^{n-d} \\
Z(Q,T)\cdot(1+(S-1)Q)^n & = & \ldots+\frac{\chi_M(S,T)-S^n}{T-1}Q^{n-d}+\ldots
\end{eqnarray*}
\end{proof}
Proposition \ref{zeta-MDS} and Theorem \ref{MDS-basis} have a direct analogue for matroids. Let $X_{n,d}$ be the coboundary polynomial of the uniform matroid on $n$ elements with rank $n-d+1$.
\begin{proposition}\label{zeta-uniform}
A matroid is uniform if and only if $P_M(Q,T)=1$.
\end{proposition}
\begin{theorem}\label{uniform-basis}
The zeta polynomial gives us a way to write the coboundary polynomial on a basis of coboundary polynomials of uniform matroids:
\[ \chi_M(S,T)=P_0(T)\,X_{n,d}+P_1(T)\,X_{n,d+1}+\ldots+P_r(T)\,X_{n,d+r}. \]
\end{theorem}

We need some more properties of the zeta polynomial. The proofs are similar to the case of the one-variable zeta polynomial as treated in \cite{duursma:2004}.
\begin{proposition}
The degree of $P_M(Q,T)$ in $S$ is $n-d-d^*+2$.
\end{proposition}
\begin{proof}
Assume that $P_r(T)$ is not zero and apply Theorem \ref{uniform-basis} to the dual matroid $M^*$. This expression starts with $X_{n,d^*+r}^*$. Since the dual of the uniform matroid is again a uniform matroid, we have $X_{n,d^*+r}^*=X_{n,n-d+2+r}$. So $n+2-d-r=d^*$ and hence $r=n-d-d^*+2$.
\end{proof}
\begin{proposition}\label{zeta-dual}
For the two-variable zeta polynomial of a matroid $M$ and dual $M^*$ we have
\[ P_{M^*}(Q,T)=P_M\left(\frac{1}{TQ},T\right)\, T^{n-k+1-d}Q^{n-d-d^*+2}. \]
\end{proposition}
\begin{proof}
Apply Theorem \ref{chi-dual} to the expression in Theorem \ref{uniform-basis}. This gives that $\chi_{M^*}(S,T)$ is equal to
\[ \begin{array}{cl}
 & (S-1)^nT^{-k}\,\chi_M\left(\frac{S+T-1}{S-1},T\right) \\
= & (S-1)^nT^{-k}\left( P_0(T)\,X_{n,d}\left(\frac{S+T-1}{S-1},T\right)+\ldots+P_r(T)\,X_{n,d+r}\left(\frac{S+T-1}{S-1},T\right) \right) \\
= & T^{-k}\left( P_r(T)\,T^{n-d-r+1}\,X_{n,n-d+2-r}+\ldots+P_0(T)\,T^{n-d-1}\,X_{n,n-d+2}\right)
\end{array} \]
and the Proposition follows.
\end{proof}

We are now ready to give an alternative proof of Theorem \ref{bound} using the two-variable zeta polynomial.

\begin{proof}[Theorem \ref{bound}]
Our goal is to determine all the coefficients $P_j(T)$ of the two-variable zeta polynomial, and thus the coboundary polynomial $\chi_M(S,T)$. Denote the coefficient of $S^j$ in $X_{n,d}$ by $X_{n,d,j}$. We know the exact value of these coefficients, just like we know the extended weight enumerator of MDS codes:
\[ \chi_j(T)={n\choose j}(T-1)\sum_{t=0}^{n-j-d}(-1)^t{n-j-1\choose t}T^{n-j-d-t}. \]
So we can split up Theorem \ref{uniform-basis} in $n+1$ equations:
\[ \chi_j(T)=\sum_{i=0}^{n-d-d^*+2}P_i(T)\,X_{n,d+i,j},\qquad j=0,\ldots,n. \]
Not all of these equations are helpful in determining the $P_i(T)$. For $j<d^*-1$ and $j>n-d$ the $\chi_j(T)$ are known by Theorem \ref{mu-chi}. In the case $n-d<j<n$ we have $\chi_j(T)=0$ and also $X_{n,d+i,j}=0$ for all $i$, so the corresponding equations just state $0=0$. For $d^*-1\leq j\leq n-d$ we don't know $\chi_j(T)$, so these equations are also not helpful. We are left with the equations for $j<d^*-1$ and $j=n$, so $d^*$ equations in the $n-d-d^*+3$ unknown $P_i(T)$. \\
We can do the same for the dual matroid, leading to $d$ equations in the $n-d-d^*+3$ unknown $P_i^*(T)$. From Proposition \ref{zeta-dual} it follows that
\[ P_i^*(T)=T^{i-k-1+d^*}P_{n-d-d^*+2-i}(T), \]
so we can replace the $P_i^*(T)$ one-to-one by the appropriate $P_i(T)$. So all together, we have $d+d^*$ equations in $n-d-d^*+3$ unknown $P_i(T)$. To get at least as many equations as unknowns, we need
\begin{eqnarray*}
d+d^* & \geq & n-d-d^*+3 \\
2(d+d^*) & \geq & n+3.
\end{eqnarray*}
This is the same bound we already obtained in Theorem \ref{bound}.
\end{proof}

\section{Open questions}

We have seen two methods to determine the coboundary polynomial $\chi_M(S,T)$ of a matroid from the M\"obius polynomials $\mu_M(S,T)$ and $\mu_{M^*}(S,T)$ of a matroid and its dual. Both methods rely on duality relations, for, respectively, the coboundary and Tutte polynomial. \\
The logical question is now: how sharp is the bound in Theorem \ref{bound}? To look for an example to show the bound is tight, we need two matroids with the same parameters and $2(d+d^*)<n+3$ that have equal M\"obius polynomials $\mu_M(S,T)$ and $\mu_{M^*}(S,T)$ but different coboundary polynomial $\chi_M(S,T)$. The smallest case is $d=d^*=3$ (because otherwise the matroid is not simple) and thus $n=10$. \\
An exhaustive computer search on 260 random matrices with the desired parameters and $k=5$ did not lead to such an example. So there is room for improvement on the Main Question. \\

In Proposition 6.3 of \cite{brylawski:1980a} the issue is addressed how many Tutte polynomials there are, given the size and rank of a matroid. This is done by looking at the affine space generated by the coefficients of the Tutte polynomial, and determining its dimension. It would be interesting to see if we can do the same thing for the M\"obius polynomial, given $n$, $k$, $d$ and $d^*$. If we determine the dimension of the affine space generated by the coefficients of the M\"obius polynomial of a matroid and its dual, we can compare it to the dimension for the Tutte polynomial. This could give us more information about the Main Question in general.

\section*{Acknowledgment}

The author would like to thank Ruud Pellikaan for stating the Main Question, and for valuable comments on this paper.

\bibliographystyle{plain}
\bibliography{wtenum}

\end{document}